\documentclass{gIPE2e_arxiv}
\usepackage{lmodern}
\usepackage{subfigure}
\usepackage{mathtools}
\usepackage{tabularx}
\usepackage{enumerate}
\newcolumntype{Y}{>{\centering\arraybackslash}X}
\newcolumntype{?}{!{\vrule width 1.2pt}}

\theoremstyle{plain}
\newtheorem{theorem}{Theorem}[section]

\newtheorem{lemma}[theorem]{Lemma}
\newtheorem{proposition}[theorem]{Proposition}

\theoremstyle{remark}
\newtheorem{remark}[theorem]{Remark}

\theoremstyle{definition}

\numberwithin{equation}{section} 
\numberwithin{figure}{section}
\numberwithin{table}{section}

\DeclarePairedDelimiter{\norm}{\lVert}{\rVert} 
\DeclarePairedDelimiter{\abs}{\lvert}{\rvert} 
\DeclarePairedDelimiter{\inner}{\langle}{\rangle}

\newcommand{\sref}[1]{section~\ref{#1}}
\newcommand{\fref}[1]{figure~\ref{#1}}

\newcommand{\Fref}[1]{Figure~\ref{#1}}
\newcommand{\absm}[1]{\abs{#1}}

\newcommand{\innerm}[1]{\inner{#1}}
 
\newcommand{\spara}[1]{\left[#1\right]}

\newcommand{\rum}[1]{\mathbb{#1}}
\newcommand{\nd}{\mathcal{R}}
\newcommand{\mta}{\mathcal{M}_a}
\newcommand{\jhalf}{J_a^{1/2}}
\newcommand{\doroverline}[2]{\overline{#1#2}}
\newcommand{\roverline}[1]{\mathpalette\doroverline{#1}}
\newcommand{\essinf}{\mathop{\textup{ess}\,\textup{inf}}}
\newcommand{\ident}{\mathop{\textup{Id}}}
\newcommand{\sang}{\zeta}

\DeclareMathOperator{\mspan}{span}

\begin{document}



\title{Comparison of linear and non-linear monotonicity-based shape reconstruction using exact matrix characterizations}

\author{Henrik Garde$^\ast$\thanks{$^\ast$Corresponding author. Email: hgar@dtu.dk
\vspace{6pt}} \\\vspace{6pt}  {\em{Department of Applied
  Mathematics and Computer Science, Technical University of Denmark,
 2800 Kgs. Lyngby, Denmark}}
\\\received{March 2016} }

\maketitle

\begin{abstract}
	Detecting inhomogeneities in the electrical conductivity is a special case of the inverse problem in electrical impedance tomography, that leads to fast direct reconstruction methods. One such method can, under reasonable assumptions, exactly characterize the inhomogeneities based on monotonicity properties of either the Neumann-to-Dirichlet map (non-linear) or its Fr\'echet derivative (linear). We give a comparison of the non-linear and linear approach in the presence of measurement noise, and show numerically that the two methods give essentially the same reconstruction in the unit disk domain. For a fair comparison, exact matrix characterizations are used when probing the monotonicity relations to avoid errors from numerical solution to PDEs and numerical integration. Using a special factorization of the Neumann-to-Dirichlet map also makes the non-linear method as fast as the linear method in the unit disk geometry. 
\end{abstract}

\begin{keywords}
Electrical impedance tomography; monotonicity method; inverse boundary value problem; ill-posed problem; direct reconstruction
\end{keywords}

\begin{classcode}65N21; 35R30; 35Q60; 35R05\end{classcode}

\section{Introduction} \label{sec:intro}

In electrical impedance tomography (EIT) the internal electrical conductivity $\gamma$ inside a bounded Lipschitz domain, which in this paper is the unit disk domain $\rum{D}$ in $\rum{R}^2$, is determined from boundary current-voltage measurements recorded through surfaces electrodes. The underlying mathematical problem, also known as the Calder\'on problem \cite{Calder'on1980}, is an inverse problem. A common mathematical formulation of EIT is the \emph{continuum model}
\begin{equation}
	\nabla \cdot(\gamma\nabla u) = 0 \text{ in } \rum{D}, \quad \nu\cdot \gamma\nabla u = g \text{ on } \partial\rum{D}, \quad \int_{\partial\rum{D}} u\, ds = 0, \label{eq:pde}
\end{equation}
where $u$ is the internal electrical potential, $\nu$ is an outwards pointing unit normal, and $g$ is the applied current density. The latter condition in \eqref{eq:pde} is a grounding of the total electrical potential at the boundary. If $\gamma\in L_+^\infty(\rum{D})$ and $g\in H^{-1/2}_\diamond(\partial\rum{D})$ with
\begin{align*}
	L^\infty_+(\rum{D}) &\equiv \left\{w\in L^\infty(\rum{D}) : \essinf w > 0 \right\}, \\
	H_\diamond^{-1/2}(\partial\rum{D}) &\equiv \left\{w\in H^{-1/2}(\partial\rum{D}) : \innerm{w,1}_{H^{-1/2},H^{1/2}} = 0 \right\},
\end{align*}
then standard elliptic theory gives rise to a unique solution $u\in H^1_\diamond(\rum{D})$ to \eqref{eq:pde}, where the $\diamond$-symbol implies functions with zero mean on the boundary
\begin{align*}
	H_\diamond^{1/2}(\partial\rum{D}) &\equiv \left\{ w\in H^{1/2}(\partial\rum{D}) : \int_{\partial\rum{D}} w \, ds = 0 \right\}, \\
	H_\diamond^1(\rum{D}) &\equiv \left\{ w\in H^1(\partial\rum{D}) : w|_{\partial\rum{D}}\in H_\diamond^{1/2}(\partial\rum{D}) \right\}.
\end{align*}
The forward problem of EIT is the map $\nd$ from the conductivity $\gamma$ to the Neumann-to-Dirichlet (ND) map $\nd(\gamma) : \nu\cdot \gamma\nabla u \mapsto u|_{\partial\rum{D}}$. The ND map relates any applied current to the corresponding boundary potential. In general $\nd(\gamma)$ is a map from $H_\diamond^{-1/2}(\partial\rum{D})$ to $H_\diamond^{1/2}(\partial\rum{D})$, however in this paper it suffices to restrict it to a map in $\mathcal{L}(L^2_\diamond(\partial\rum{D}))$, the space of linear and bounded operators from $L^2_\diamond(\partial\rum{D})$ to itself, where
\begin{equation*}
	L^2_\diamond(\partial\rum{D}) \equiv \left\{ w\in L^2(\partial\rum{D}) : \int_{\partial\rum{D}} w\,ds = 0  \right\}.
\end{equation*}
In this sense $\nd(\gamma)$ is both compact and self-adjoint in the usual $L^2(\partial\rum{D})$-inner product, which will be denoted by $\innerm{\cdot,\cdot}$. The inverse problem of EIT is from knowledge of $\nd(\gamma)$ to reconstruct $\gamma$. Uniqueness has been shown with various regularity assumptions depending on the dimension $d$ \cite{Sylvester1987,Nachman1988a,Novikov1988,Nachman1996,Haberman_2013} and for $d=2$ there is uniqueness for general $L^\infty_+$-conductivities when the domain is simply connected \cite{Astala2006a}.

The inverse problem of EIT is severely ill-posed and with reasonable assumptions it is only possible to get conditional logarithmic stability  \cite{Alessandrini1988,Mandache2001}. It is therefore not always of interest to perform a full reconstruction of $\gamma$, but rather reconstruct inclusions/inhomogeneities from a known or uninteresting background, which is an easier problem. Here it is assumed that 
\begin{equation}
	\gamma \equiv 1 + \kappa \chi_\mathcal{D}, \label{gammaD}
\end{equation}
where $\chi_\mathcal{D}$ is a characteristic function over the sought inclusion $\mathcal{D}$ with $\overline{\mathcal{D}}\subset \rum{D}$, on which the background conductivity $1$ is perturbed by $\kappa\in L^\infty_+(\rum{D})$. Direct reconstruction methods for such inclusion detection are prominently the factorization method \cite{Bruhl2001,Bruhl2000,Lechleiter2008a} and the enclosure method \cite{Ikehata1999a,Ikehata2000c}. In this paper we investigate the more recent monotonicity method \cite{Harrach13,GardeStaboulis_2016,Harrach15,Tamburrino2006,Tamburrino2002} that makes use of a monotonicity property of $\gamma\mapsto\nd(\gamma)$. The basic idea of the method is to determine whether or not a chosen ball $B$ is inside the inclusion $\mathcal{D}$. For instance in the simple case $\gamma \equiv 1 + \chi_\mathcal{D}$, then
\begin{equation}
	B\subseteq \mathcal{D} \quad\text{implies}\quad \nd(1+\chi_B)-\nd(1+\chi_\mathcal{D}) \geq 0, \label{nonlinapp}
\end{equation}
where the inequality is in terms of positive semi-definiteness. Checking the positive semi-definiteness for all balls in $\rum{D}$ gives an upper bound on $\mathcal{D}$, and in \cite{Harrach13} it was shown that it completely characterizes $\overline{\mathcal{D}}$ if $\mathcal{D}$ has connected complement. 

The map $\gamma\mapsto\nd(\gamma)$ is non-linear and thus the evaluation of $\nd(1+\chi_B)$ is costly as each evaluation requires solving \eqref{eq:pde} for several Neumann conditions. In \cite{Harrach13} it was shown that, without loss of shape information, the non-linear part could be replaced by a linearization
\begin{equation}
	B\subseteq \mathcal{D} \quad\text{implies}\quad \nd(1) + \tfrac{1}{2}\nd'(1)\chi_B- \nd(1+\chi_\mathcal{D})\geq 0. \label{linapp}
\end{equation}
Using the Fr\'echet derivative is attractive as it only requires one evaluation of the derivative which can often be determined beforehand. 

In this paper we compare reconstructions based on the non-linear approach \eqref{nonlinapp} and the linear approach \eqref{linapp}, which in the noiseless case solve the exact inverse problem of reconstructing $\overline{\mathcal{D}}$. However, it is still unknown if the two methods are equally noise robust. In \cite{Harrach15} resolution bounds for stable reconstruction are determined, which for the linear method are more pessimistic, though the bounds are not shown to be optimal. With various levels of noise added to the measurements, the numerical examples in \sref{sec:numerical} surprisingly show that there is essentially no difference in the reconstructions based on the non-linear and the linear approach for the chosen examples.

For a fair comparison of the non-linear and linear method, exact matrix representations are determined for $\nd(1+\beta\chi_B)$ and $\nd'(1)\chi_B$ for any ball $B$ in $\rum{D}$, in order to avoid errors from numerical solution to PDEs and numerical integration. In this specific geometry the non-linear method furthermore becomes as fast as the linear method, by use of an explicit factorization derived from M\"obius transformations.

More precise forward models for EIT exist for practical measurements, such as the \emph{complete electrode model} (CEM). It was recently proved in \cite{GardeStaboulis_2016} that the monotonicity method can be regularized against noise and generalizes to various approximations of the continuum model, including the CEM. By simply replacing the ND map with the CEM counterpart gives a reconstruction that is interlaced between two reconstructions from the continuum model; one without regularization and one with regularization. Thus, in this sense, the comparison made here also directly applies to the CEM variant of the monotonicity method.

The contents of this paper are organized as follows: in \sref{sec:monorecon} the monotonicity method is outlined and M\"obius transformations are introduced to relate non-concentric ball inclusions to concentric ones. In \sref{sec:matstruct} the exact matrix representations of the ND map and its Fr\'echet derivatives are derived, and their matrix structures are elaborated on. Implementation details and numerical examples are given in \sref{sec:numerical}, and finally we conclude in \sref{sec:conclusions}.

\section{Monotonicity-based shape reconstruction} \label{sec:monorecon}

In this section we shortly outline the regularized monotonicity method in the linear and non-linear case. For $\kappa\in L^\infty_+(\rum{D})$ and open set $\mathcal{D}\subset \rum{D}$, define the conductivity as in \eqref{gammaD} where $\chi_\mathcal{D}$ is a characteristic function on $\mathcal{D}$. We denote for $\beta>0$ the monotonicity-based reconstructions by the following sets of open balls:
\begin{align*}
	\mathcal{T} &\equiv \left\{ B\subseteq \rum{D} \text{ open ball} : \nd(1+\beta\chi_B) -\nd(\gamma)\geq 0  \right\}, \\
	\mathcal{T}' &\equiv \left\{ B\subseteq \rum{D} \text{ open ball} : \nd(1) + \beta\nd'(1)\chi_B - \nd(\gamma)\geq 0 \right\}, 
\end{align*}
where $\cup \mathcal{T}$ and $\cup \mathcal{T}'$ are the unions of the above ball-collections, which will be directly compared to the inclusion $\mathcal{D}$. Here the Fr\'echet derivative $\nd'(1)$ of $\gamma\mapsto\nd(\gamma)$ evaluated at $1$ and in direction $\eta\in L^\infty(\rum{D})$ is given by
\begin{equation}
	\innerm{\nd'(1)[\eta]f,g} = -\int_{\rum{D}} \eta \nabla w_f\cdot \overline{\nabla w_g}\,dx, \label{eq:frechet}
\end{equation}
where $w_f$ and $w_g$ are solutions to \eqref{eq:pde} with conductivity $1$ and Neumann condition $f$ and $g$, respectively.

Using the formulation in \cite[Section 2.2]{GardeStaboulis_2016}, if we assume that $\overline{\mathcal{D}}\subseteq \rum{D}$ and $\mathcal{D}$ has connected complement (no holes in the inclusions), then
\begin{align}
	\mathcal{D}\subseteq \cup \mathcal{T} \subseteq \overline{\mathcal{D}} &\quad \text{if} \quad 0<\beta\leq \essinf\kappa, \label{eq:cupT}\\
	\mathcal{D}\subseteq \cup \mathcal{T}' \subseteq \overline{\mathcal{D}} &\quad \text{if} \quad 0<\beta\leq \essinf\left(\frac{\kappa}{1+\kappa}\right). \label{eq:cupT'}
\end{align} 

\begin{remark}
	Since $\kappa\in L^\infty_+(\rum{D})$ then $\essinf \kappa > 0$, which ensures that the bounds presented in \cite{Harrach13} known as the inner and outer support of $\kappa$ are given by $\mathcal{D}$ and $\overline{\mathcal{D}}$, respectively. Furthermore, it allows the use of a single $\beta$-value for the test inclusions (rather than a union for all $\beta$-values) and avoids any further regularity assumptions on $\kappa$. See \cite{GardeStaboulis_2016} for details on the derivation in this setting, and also in particular Examples 4.2 and 4.4 in \cite{Harrach13}.
\end{remark}

Assuming prior knowledge of a lower bound on the perturbation $0<\beta^\textup{L} \leq \kappa$, and writing
\begin{equation*}
	\frac{\kappa}{1+\kappa} = 1-\frac{1}{1+\kappa},
\end{equation*}
then an admissible choice of the $\beta$-value in \eqref{eq:cupT} and \eqref{eq:cupT'} can be guaranteed by
\begin{equation}
	\beta^\textup{nonlin} \equiv \beta^\textup{L}, \qquad \beta^\textup{lin} \equiv 1-\frac{1}{1+\beta^{\textup{L}}} =  \frac{\beta^\textup{L}}{1+\beta^\textup{L}}. \label{eq:betaval}
\end{equation} 
The main advantage of the linear method is that $\nd'(1)$ can be evaluated cheaply and prior to reconstruction. While the $\nd(1+\beta\chi_B)$-maps can also be evaluated prior to reconstruction it requires knowledge of the $\beta$-value beforehand, and different $\beta$-values require new evaluations. 

Given a compact noise perturbation $E^\delta\in \mathcal{L}(L^2_\diamond(\partial\rum{D}))$ with $\norm{E^\delta}_{\mathcal{L}(L^2_\diamond(\partial\rum{D}))} \leq \delta$, then the noisy datum $\nd^\delta(\gamma)$ is modelled with additive noise
\begin{equation}
	\nd^\delta(\gamma) \equiv \nd(\gamma) + E^\delta. \label{eq:ndnoisy}
\end{equation}
For regularization parameter choice $\alpha(\delta) \geq \delta$ with $\lim_{\delta\to 0}\alpha(\delta) = 0$ it was proved in \cite[Theorem 1]{GardeStaboulis_2016} that the following regularized reconstructions give upper bounds $\cup \mathcal{T}_\alpha$ and $\cup\mathcal{T}_\alpha'$ of $\mathcal{D}$, and that they converge as the noise level tends to zero, $\delta\to 0$:
\begin{align}
	\mathcal{T}_\alpha &\equiv \left\{ B\subseteq \rum{D} \text{ open ball} : \nd(1+\beta\chi_B) + \alpha\ident - \nd^\delta(\gamma)\geq 0  \right\}, \label{eq:regnonlinrecon}\\
	\mathcal{T}'_\alpha &\equiv \left\{ B\subseteq \rum{D} \text{ open ball} : \nd(1) + \beta\nd'(1)\chi_B + \alpha\ident - \nd^\delta(\gamma) \geq 0 \right\}. \label{eq:reglinrecon}
\end{align}

The background conductivity of $1$ is merely for ease of presentation, and other (constant) background conductivities can be used with the identity
\begin{equation*}
	\nd(c\gamma) = \frac{1}{c}\nd(\gamma),\enskip c>0.
\end{equation*}

\subsection{M\"obius transformations and additional notation}

To get a precise and fast evaluation of $\nd(1+\beta\chi_B)$, we use M\"obius transformations to move non-concentric balls $B_{C,R}$ to concentric balls $B_{0,r}$, and abuse that the spectrum of $\nd(1+\beta\chi_{B_{0,r}})$ is known. Note that this section is largely based on \cite{Garde2016}. To shorten notation we will use the abbreviation 
\begin{equation*}
	\gamma_{C,R} \equiv 1 + \beta\chi_{B_{C,R}}, 
\end{equation*}
where $B_{C,R}$ is an open ball with centre $C$ and radius $R$, and we will throughout the paper identify $(x_1,x_2)\in\rum{R}^2$ with $x_1+ix_2\in\rum{C}$. Denote by $M_a$, for any $a\in\rum{D}$, the M\"obius transformation
\begin{equation*}
	M_a(x) \equiv \frac{x-a}{\overline{a}x-1},\enskip x\in\rum{D}.
\end{equation*}
Here $M_a : \rum{D}\to\rum{D}$ and $\partial\rum{D}\to\partial\rum{D}$, and is furthermore an involution i.e.\ its own inverse $M_a^{-1} = M_a$. Let $a \equiv \rho e^{i\sang}$ for $0\leq \rho < 1$ and $\sang\in\rum{R}$. For $0<r<1$ then $M_a(B_{0,r}) = B_{C,R}$ with
\begin{equation*}
	C = \frac{\rho(r^2-1)}{\rho^2r^2-1}e^{i\sang},\quad R = \frac{r(\rho^2-1)}{\rho^2r^2-1}.
\end{equation*}
Furthermore, for $C = ce^{i\sang}$ with $0\leq c<1$ and $0<R<1-c$, then there is a unique $a\in\rum{D}$ such that $M_a(B_{C,R}) = B_{0,r}$ where $r$ and $a$ satisfies
\begin{equation*}
	r = \frac{1+R^2-c^2-\sqrt{((1-R)^2-c^2)((1+R)^2-c^2)}}{2R}, \quad a = \frac{C}{1-Rr}.
\end{equation*}
The above notation for the relation between a non-concentric ball $B_{C,R}$ and a concentric ball $B_{0,r}$ will be used in the remainder of the paper, and it is also illustrated in \fref{fig:fig0}.
\begin{figure}[htb]
\centering
\includegraphics[width = .8\textwidth,natwidth=8.5in,natheight=3.33in]{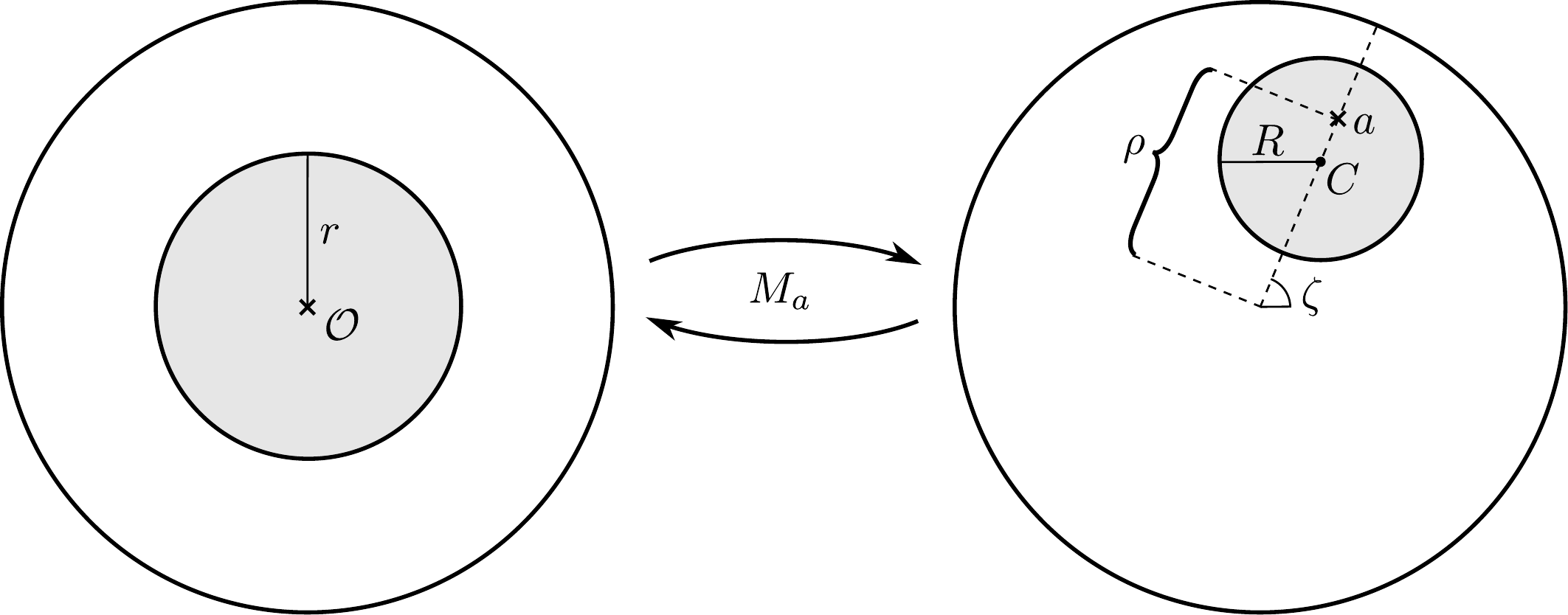}
\caption{Illustration of $M_a$ acting on balls $B_{0,r}$ and $B_{C,R}$.} \label{fig:fig0}
\end{figure}

For a function $f:\rum{D}\to \rum{C}$ (or $\partial\rum{D}\to\rum{C}$) we define $\mta f \equiv f\circ M_a$, i.e.\ the transformation $\mta$ applying the change of variables $M_a$. The Jacobian determinant for the change of variables on $\rum{D}$ is
\begin{equation*}
	J_a(x) \equiv \left(\frac{1-\rho^2}{\absm{\overline{a}x -1}^2} \right)^2, \enskip x\in\rum{D}.
\end{equation*}
The Jacobian determinant for the integral on $\partial\rum{D}$ is $J_a^{1/2}|_{\partial\rum{D}}$, and there is the following factorization of $\nd(\gamma_{C,R})$ (cf.\ \cite[Appendix B]{Garde2016})
\begin{equation}
	\nd(\gamma_{C,R}) = P\mta \nd(\gamma_{0,r})\jhalf\mta. \label{ndmob}
\end{equation}
Here $P : L^2(\partial\rum{D})\to L_\diamond^2(\partial\rum{D})$ is the orthogonal projection given by $Pf \equiv f-\frac{1}{2\pi}\int_{\partial\rum{D}} f\,ds$, and $J_a^{1/2}:L^2(\partial\rum{D})\to L^2(\partial\rum{D})$ is the multiplication operator $f\mapsto J_a^{1/2}|_{\partial\rum{D}}f$. 

The operator $\nd(\gamma_{0,r})$ for a concentric ball has the Fourier basis 
\begin{equation}
	f_n(\theta) \equiv \frac{1}{\sqrt{2\pi}}e^{in\theta},\enskip n\in\rum{Z}\setminus\{0\}, \label{eq:fbasis}
\end{equation}
as eigenfunctions with eigenvalues
\begin{equation}
	\lambda_n \equiv \frac{2+\beta(1-r^{2\absm{n}})}{2+\beta(1+r^{2\absm{n}})}\frac{1}{\absm{n}},\enskip n\in\rum{Z}\setminus\{0\}. \label{eq:lam}
\end{equation}
So the factorization in \eqref{ndmob} implies that the $\beta$-dependence of $\nd(\gamma_{C,R})$ is given explicitly through a diagonalization of $\nd(\gamma_{0,r})$ with the eigenvalues \eqref{eq:lam}. Here $P\mta$ and $\jhalf\mta$ only depend on the transformation parameter $a$, and can be determined prior to reconstruction, thus making the non-linear and linear methods have identical computational complexity.

In order to determine matrix representations of $P\mta$ and $\jhalf\mta$ in the Fourier basis it is relevant to investigate the action of $M_a$ on a trigonometric function. Note in particular that for $f(x) = x^n$ with $x=x_1+ix_2$, then $\mta f = M_a(x)^n = M_a(e^{i\theta})^n$ on $\partial\rum{D}$. Utilizing that $M_a$ maps $\partial\rum{D}$ to itself:
\begin{equation}
	M_a(e^{in\theta}) = e^{in\psi_a(\theta)} = M_a(e^{i\theta})^n =  \left(\frac{e^{i\theta}-\rho e^{i\sang}}{\rho e^{i(\theta-\sang)}-1}\right)^n, \label{eq:mobexp}
\end{equation}
where $\psi_a$ is the corresponding transformation in the angular variable on $\partial\rum{D}$.

\begin{lemma} \label{lemma:psia}
	The map $\psi_a$ in \eqref{eq:mobexp} is given by
	\begin{equation*}
		\psi_a(\theta) \equiv \pi + \sang + 2\arctan \left( \frac{1+\rho}{1-\rho}\tan \left( \frac{\theta-\sang}{2} \right) \right).
	\end{equation*}
\end{lemma}
\begin{proof}
	By standard trigonometric identities it follows that
	\begin{equation*}
		e^{iz} = \frac{1+i\tan(z/2)}{1-i\tan(z/2)}, \enskip z\in(-\pi,\pi),
	\end{equation*}
	thus
	\begin{equation}
		e^{i(\pi + 2\arctan(z))} = \frac{-1-iz}{1-iz},\enskip z\in [-\infty,\infty]. \label{eq:exparctan}
	\end{equation}
	Note that \eqref{eq:exparctan} is defined on the extended real line, since both sides of the equation converge to $1$ in the limit $z\to\pm\infty$.
	
	Since $M_a(e^{i\theta}) = e^{i\sang}M_\rho(e^{i(\theta-\sang)})$ we get 
	\begin{equation}
		\psi_a(\theta) = \sang + \psi_\rho(\theta-\sang), \label{eq:psiarho}
	\end{equation}
	so it is sufficient to consider $\sang = 0$. Now assume that $\psi_\rho(\theta)$ can be written on the form \eqref{eq:exparctan}, then
	\begin{equation}
		e^{i\psi_\rho(\theta)} = \frac{e^{i\theta}-\rho}{\rho e^{i\theta}-1} = \frac{-1-iz}{1-iz} \Rightarrow z = \frac{(1+\rho)(1-e^{i\theta})}{(1-\rho)(1+e^{i\theta})}i = \frac{1+\rho}{1-\rho}\tan(\theta/2). \label{eq:psirho}
	\end{equation}
	Since $z\in [-\infty,\infty]$ in \eqref{eq:psirho} then the assumption that $\psi_\rho(\theta)$ is of the form \eqref{eq:exparctan} is valid. Now combining \eqref{eq:psirho} with \eqref{eq:exparctan} and \eqref{eq:psiarho} yields the desired result.
\end{proof}

\section{Matrix structures and characterizations} \label{sec:matstruct}

For a general real-valued $\gamma\in L^\infty_+(\rum{D})$ it holds that $\nd(\gamma)\overline{f} = \overline{\nd(\gamma)f}$, which is a consequence of $\nd(\gamma)$ being linear and $u$ in \eqref{eq:pde} being real-valued when $\gamma$ and the Neumann condition are real-valued. Defining $\{f_n\}_{n\in\rum{Z}\setminus\{0\}}$ as the usual Fourier basis for $L^2_\diamond(\partial\rum{D})$ as in \eqref{eq:fbasis} gives the identity $\nd(\gamma)f_n = \overline{\nd(\gamma)f_{-n}}$, i.e.
\begin{equation}
	\mathcal{A}_{n,m} \equiv \innerm{\nd(\gamma)f_m,f_n} = \overline{\innerm{\nd(\gamma)f_{-m},f_{-n}}}=\overline{\mathcal{A}_{-n,-m}}, \enskip n,m\in\rum{Z}\setminus\{0\}. \label{centroherm1}
\end{equation}
So by arranging the row and column indices in the matrix representation $\mathcal{A}$ from negative (top left) to positive (bottom right) gives a \emph{centrohermitian} matrix \cite{Pressman1998}, meaning that there is symmetry (similar to a Hermitian matrix) across the centre of the matrix. The centrohermitian property can be written as in \eqref{centroherm1} for that particular choice of indices (which will be used throughout this paper), or in general as
\begin{equation*}
	\mathcal{A} = \mathcal{J}\overline{\mathcal{A}}\mathcal{J}, 
\end{equation*}
where $\mathcal{J}$ is the \emph{exchange matrix} which has zeroes in all entries except on the anti-diagonal (from bottom left to top right) where its entries equal $1$. Since $\mathcal{R}(\gamma)$ is self-adjoint also makes $\mathcal{A}$ Hermitian in addition to being centrohermitian.

In \cite[Appendix B]{Garde2016} an explicit matrix representation of $\nd(\gamma_{C,R})$ was determined by the use of basis functions orthonormal in weighted $L^2$-inner products. This matrix representation cannot be used for the monotonicity method as the basis functions depend on the transformation $M_a$, and here we need a fixed basis, namely the same used for the datum $\nd(\gamma)$. A common choice of orthonormal basis for the unit disk is the Fourier basis \eqref{eq:fbasis} for $L^2_\diamond(\partial\rum{D})$, which will also be used here. Now a matrix representation of $P\mta$ and $\jhalf\mta$ from \eqref{ndmob} can be found in terms of the Fourier basis. 
\begin{theorem} \label{H:thm}
	Recall that $a = \rho e^{i\sang}$ and define the matrix $H_a$ by
	\begin{equation}
		(H_a)_{n,m} \equiv \inner{f_m,\mta f_n}, \enskip n,m\in\rum{Z}\setminus\{0\}, \label{Hamat}
	\end{equation}
	then $H_a$ has the following properties (note in particular that \eqref{H:complex}-\eqref{H:submat} explicitly defines the entire matrix):
	\begin{enumerate}[(i)]
		\item $H_a$ is a matrix representation of $J_a^{1/2}\mta$. \label{H:mat}
		\item $(H_a)^*$ is a matrix representation of $P\mta$. \label{H:adjointmat}
		\item $H_a$ is involutory, i.e.\ $H_a = H_a^{-1}$. \label{H:inv}
		\item $(H_a)_{n,m} = e^{i(m-n)\sang}(H_{\rho})_{n,m},\enskip\forall n,m$. \label{H:complex}
		\item $H_a$ is centrohermitian, i.e.\ $(H_a)_{n,m} = \overline{(H_a)_{-n,-m}},\enskip\forall n,m$. \label{H:flip}
		\item $H_a$ is block diagonal with $(H_a)_{n,m} = 0$ for $n<0,m>0$ and for $n>0,m<0$. \label{H:SW_NE}
		\item There is the following formula for $n>0,m>0$: \label{H:submat}
		\begin{equation}
			(H_\rho)_{n,m} = \sum_{k = \max\{n-m,0\}}^n (-1)^{n-k}\binom{k+m-1}{k+m-n}\binom{n}{k}\rho^{2k+m-n}. \label{Hkj}
		\end{equation}
	\end{enumerate}
\end{theorem}
\begin{proof}
 	Since $J_a^{1/2}$ is the Jacobian determinant for the change of variables $M_a$ on $\partial\rum{D}$, and that $\mta = \mta^{-1}$, we get
	\begin{equation*}
		(H_a)_{n,m} = \innerm{J_a^{1/2}\mta f_m,f_n},
	\end{equation*}
	which shows \eqref{H:mat}. \eqref{H:inv} follows directly from \eqref{H:mat} as $J_a^{1/2}\mta$ is an involution on $L^2_\diamond(\partial\rum{D})$ (cf.~\cite[Proposition 2.2]{Garde2016}). Furthermore, as $P$ is self-adjoint in the $L^2(\partial\rum{D})$-inner product (as it is an orthogonal projection), then
	\begin{equation*}
		\innerm{\jhalf\mta f,g} = \innerm{f,\mta g} = \innerm{Pf,\mta g} = \innerm{f,P\mta g},\enskip \forall f,g\in L^2_\diamond(\partial\rum{D}),
	\end{equation*}
	i.e.\ the adjoint of $P\mta$ (in terms of maps from $L_\diamond^2(\partial\rum{D})$ to itself) is $\jhalf\mta$, which shows \eqref{H:adjointmat}.
	
	Proof of \eqref{H:complex} and \eqref{H:flip}: from Lemma \ref{lemma:psia} and \eqref{eq:psiarho} where $\psi_\rho$ is $2\pi$-periodic, then a change of variable from $\theta$ to $\theta+\sang$ gives
	\begin{align*}
		\inner{f_m,\mta f_n} &= \frac{1}{2\pi}\int_0^{2\pi} e^{i(m\theta-n\sang-n\psi_\rho(\theta-\sang))}\,d\theta = \frac{1}{2\pi}\int_0^{2\pi} e^{i(m\theta+m\sang-n\sang-n\psi_\rho(\theta))}\,d\theta \\
		&=e^{i(m-n)\sang}\frac{1}{2\pi}\int_0^{2\pi} e^{i(m\theta-n\psi_\rho(\theta))}\,d\theta = e^{i(m-n)\sang}\inner{f_m,\mathcal{M}_\rho f_n}.
	\end{align*}
	This shows \eqref{H:complex}. Furthermore, 
	\begin{equation*}
		\inner{f_m,\mta f_n} = \frac{1}{2\pi}\int_0^{2\pi}e^{i(m\theta-n\psi_a(\theta))}\,d\theta = \overline{\frac{1}{2\pi}\int_0^{2\pi}e^{-i(m\theta-n\psi_a(\theta))}\,d\theta} = \overline{\inner{f_{-m},\mta f_{-n}}}.
	\end{equation*}
	
	Proof of \eqref{H:SW_NE}: first notice the identity
	\begin{align*}
		\overline{\mathcal{M}_\rho f_n} &= \frac{1}{\sqrt{2\pi}}\overline{\left(\frac{e^{i\theta}-\rho}{\rho e^{i\theta}-1}\right)^n} = \frac{1}{\sqrt{2\pi}}\left(\frac{\rho e^{-i\theta}-1}{e^{-i\theta}-\rho }\right)^{-n} = \frac{1}{\sqrt{2\pi}}\left(\frac{e^{i\theta}-\rho}{\rho e^{i\theta}-1}\right)^{-n} \\
		&= \frac{1}{\sqrt{2\pi}}(e^{i\theta}-\rho)^{-n}(\rho e^{i\theta}-1)^n.
	\end{align*}
	Assume $n>0$, then using the binomial theorem for both $(e^{i\theta}-\rho)^{-n}$ and $(\rho e^{i\theta}-1)^n$ (which for $(e^{i\theta}-\rho)^{-n}$ converges as $\rho<1$) and using that the negative binomial coefficient can be written as
	\[
		\binom{-n}{k'} = (-1)^{k'}\binom{n+k'-1}{k'},
	\]
	gives
	\begin{align*}
		\overline{\mathcal{M}_\rho f_n} &= \frac{1}{\sqrt{2\pi}}\spara{ \sum_{k'=0}^\infty (-1)^{k'}\binom{n+k'-1}{k'}(-\rho)^{k'} e^{-i(k'+n)\theta} }\cdot \spara{ \sum_{k=0}^n \binom{n}{k} (-1)^{n-k}\rho^{k} e^{ik\theta} }  \\
		&= \frac{1}{\sqrt{2\pi}}\sum_{k'=0}^\infty \sum_{k=0}^n (-1)^{n-k}\binom{n+k'-1}{k'}\binom{n}{k}\rho^{k+k'}e^{i(k-k'-n)\theta},\enskip n>0. 
	\end{align*}
	Thus
	\begin{equation}
		\inner{f_m,\mathcal{M}_\rho f_n} = \frac{1}{2\pi}\sum_{k'=0}^\infty \sum_{k=0}^n (-1)^{n-k}\binom{n+k'-1}{k'}\binom{n}{k}\rho^{k+k'}\int_0^{2\pi}e^{i(m+k-k'-n)\theta}\,d\theta. \label{binomsum}
	\end{equation}
	If $m<0$ and $n>0$ then $m+k-k'-n < 0$ for all $k=0,1,\dots,n$ and $k'\geq 0$ so $\int_0^{2\pi}e^{i(m+k-k'-n)\theta}\,d\theta = 0$, i.e.
	\begin{equation}
		\inner{f_m,\mathcal{M}_\rho f_n} = 0,\enskip m<0, n>0. \label{zeroint}
	\end{equation}
	Now \eqref{H:SW_NE} follows from \eqref{zeroint}, \eqref{H:complex}, and \eqref{H:flip}.
	
	Proof of \eqref{H:submat}: Assume $n>0$ and $m>0$. We have 
	\begin{equation*}
		\frac{1}{2\pi}\int_0^{2\pi}e^{i(m+k-k'-n)\theta}\,d\theta = \delta_{m+k-k'-n,0},
	\end{equation*}
	so to find the non-zero coefficients in \eqref{binomsum} we need to determine when $m+k-k'-n = 0$, i.e.\ set $k' = m+k-n$ and find $m+k-n\geq 0$ (as we have $k'\geq 0$). Since $k\geq 0$ we need $k' = m+k-n$ and $k\geq \max\{n-m,0\}$. Thus \eqref{binomsum} becomes
	\begin{equation*}
		\inner{f_m,\mathcal{M}_\rho f_n} = \sum_{k=\max\{n-m,0\}}^n (-1)^{n-k}\binom{k+m-1}{k+m-n}\binom{n}{k}\rho^{2k+m-n},\enskip n>0,m>0.
	\end{equation*}
\end{proof}

\begin{figure}[htb]
\centering
\subfigure[$\absm{a} = 0.2$]{\includegraphics[width=.32\textwidth,natwidth=5.04in,natheight=4.93in]{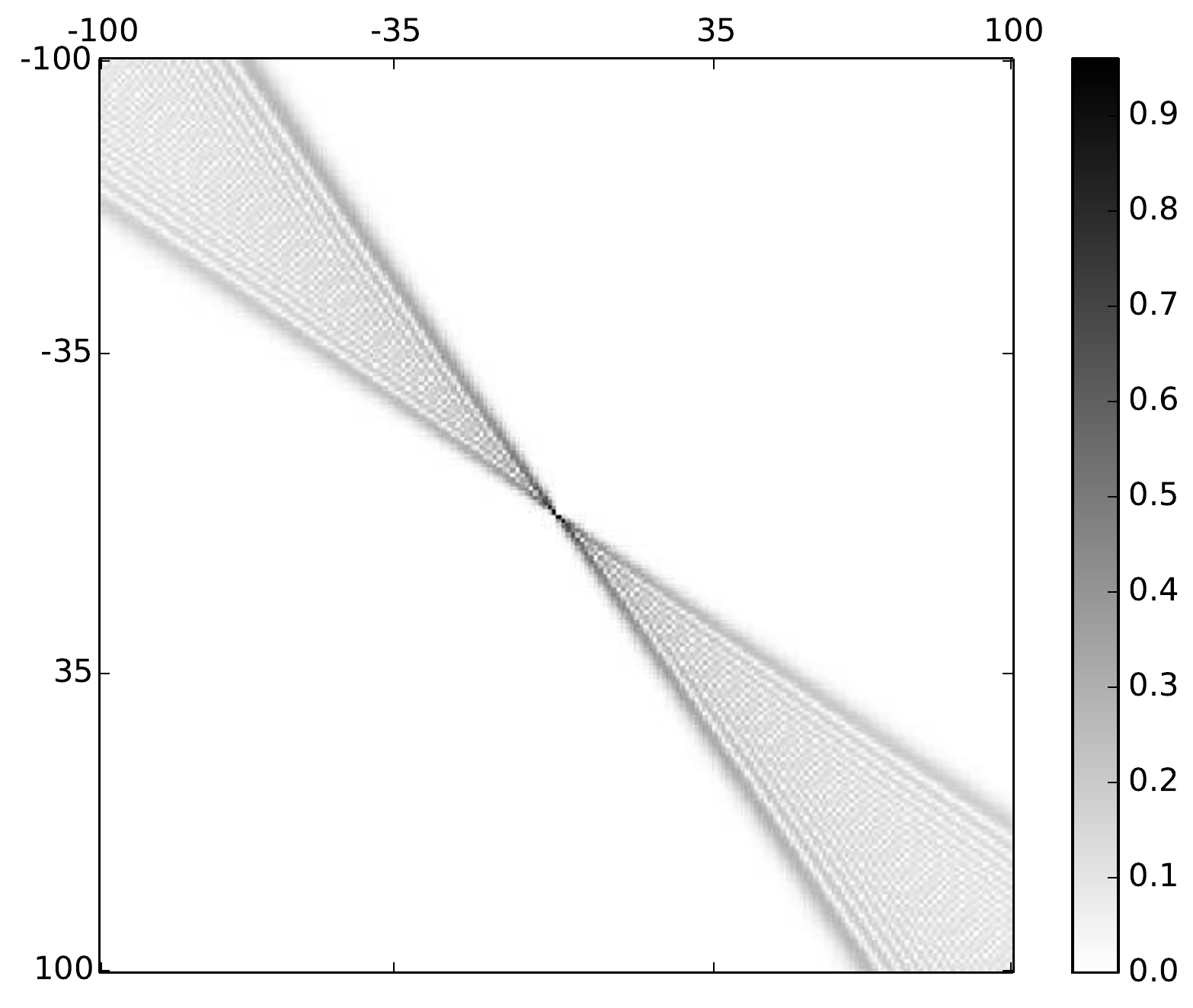}} \label{fig:fig1a}
\hfill
\subfigure[$\absm{a} = 0.5$]{\includegraphics[width=.32\textwidth,natwidth=5.13in,natheight=4.93in]{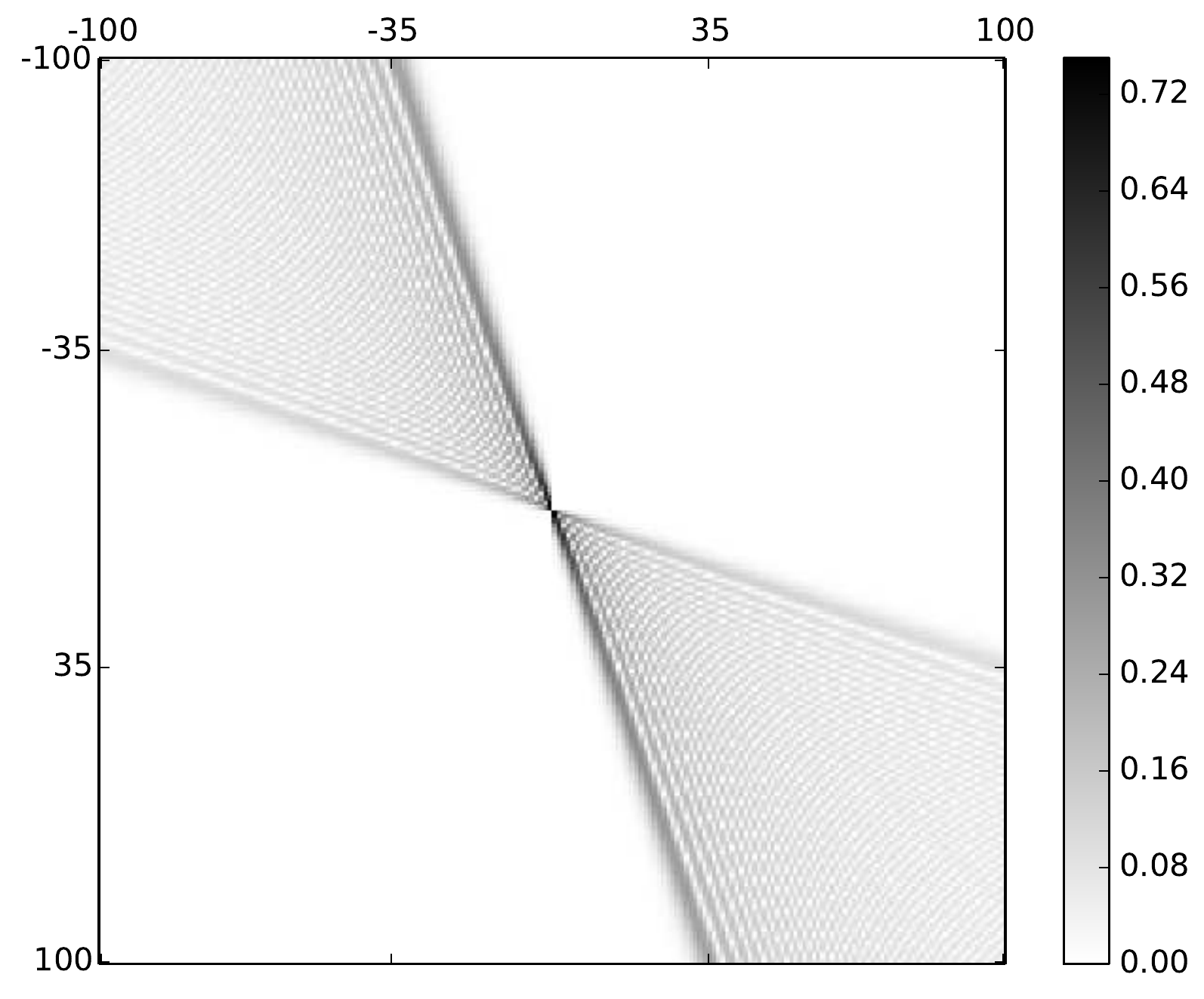}} \label{fig:fig1b}
\hfill
\subfigure[$\absm{a} = 0.8$]{\includegraphics[width=.32\textwidth,natwidth=5.04in,natheight=4.93in]{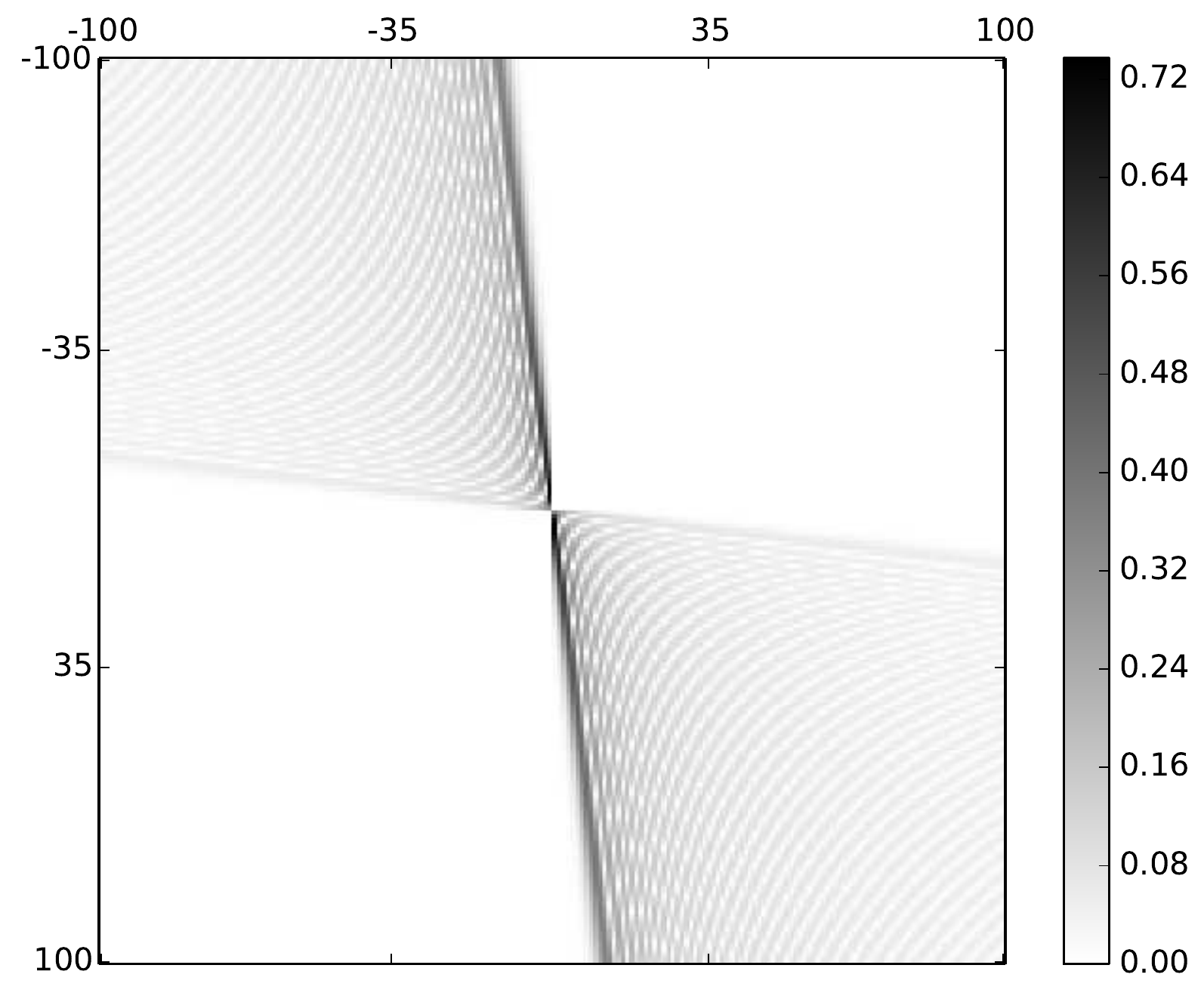}} \label{fig:fig1c}
\caption{Plot of absolute value $\absm{(H_a)_{n,m}}$ for $\absm{n},\absm{m} = 1,2,\dots,100$.} \label{fig:fig1}
\end{figure}

\Fref{fig:fig1} shows the structure of the $H_a$-matrices for different choices of $a$. It is evident that as $\absm{a}\to 1$, and thereby $M_a$ maps $B_{0,r}$ to a ball $B_{C,R}$ close to the boundary, the representation $H_a$ becomes less sparse in the Fourier basis. 

If we denote the matrix representation of $\nd(\gamma_{C.R})$ by
\begin{equation*}
	A_{n,m} \equiv \innerm{\nd(\gamma_{C,R})f_m,f_n}, 
\end{equation*}
and for $\nd(\gamma_{0,r})$ the matrix representation is a diagonal matrix $D$ with $D_{n,n} \equiv \lambda_n$ from \eqref{eq:lam}, then from Theorem \ref{H:thm} and \eqref{ndmob} we have
\begin{equation}
	A = (H_a)^* D H_a. \label{afact}
\end{equation}
Due to the centrohermitian and block diagonal properties of both $H_a$ and $D$ they have the block structure
\begin{equation}
	H_a = \begin{pmatrix}
		\mathcal{J}\roverline{H_a^+}\mathcal{J} & 0 \\ 0 & H_a^+
	\end{pmatrix}, \quad D = \begin{pmatrix}
		\mathcal{J}D^+\mathcal{J} & 0 \\ 0 & D^+
	\end{pmatrix}, \label{centrohermstruct}
\end{equation}
where $\mathcal{J}$ is the previously defined exchange matrix, $H_a^+$ is the lower right part of $H_a$ i.e.\ for $n>0,m>0$, and similarly $D^+$ is the lower right part of $D$. Therefore we get the following structure for $A$:
\begin{equation}
	A = \begin{pmatrix}
		\mathcal{J}\roverline{(H_a^+)^* D^+ H_a^+}\mathcal{J} & 0 \\
		0 & (H_a^+)^*{D^+} H_a^+
	\end{pmatrix}. \label{Amatstruct}
\end{equation}
That the matrix representation of $\nd(\gamma_{C,R})$, for any ball inclusion $B_{C,R}$, is a block diagonal matrix of the form \eqref{Amatstruct} is a non-trivial result of the factorization \eqref{afact}. For a general $\gamma\in L_+^\infty(\rum{D})$ the matrix representation of $\nd(\gamma)$ is not block diagonal. Furthermore, it also means that we only need to evaluate and save the $H_a^+$-part for assembling $A$.

An explicit formula can also be obtained for the Fr\'echet derivative $\mathcal{R}'(1)$ on ball inclusions. 
\begin{proposition}
	Denote by $A'$ the matrix representation of $\nd'(1)[\chi_{B_{C,R}}]$, i.e.
	\begin{equation}
		A'_{n,m} \equiv \innerm{\nd'(1)[\chi_{B_{C,R}}]f_m,f_n}, \enskip n,m\in\rum{Z}\setminus\{0\}, \label{a'mat}
	\end{equation}
	then 
	\begin{enumerate}[(i)]
		\item $A'$ is Hermitian, i.e.\ $A' = (A')^*$. \label{A':herm}
		\item $A'$ is centrohermitian, i.e.\ $A'_{n,m} = \overline{A'_{-n,-m}}$. \label{A':centroherm}
		\item $A'$ is block diagonal with $A'_{n,m} = 0$ for $n<0, m>0$ and for $n>0, m<0$. \label{A':blockdiag}
		\item There is the following formula for $n>0, m>0$, recalling that $C = ce^{i\sang}$: \label{A':formula}
		\begin{equation}
			A'_{n,m} = -e^{i(m-n)\sang}\sum_{k=0}^{\min\{n,m\}-1} \frac{1}{k+1}\binom{m-1}{k}\binom{n-1}{k}c^{m+n-2k-2}R^{2k+2}. \label{eq:A'formula}
		\end{equation}
	\end{enumerate}
	\begin{proof}
		\eqref{A':herm} and \eqref{A':centroherm} follows directly from \eqref{eq:frechet} and that $w_{\overline{g}} = \overline{w_g}$ (similar argument as in the beginning of \sref{sec:matstruct}), and the proof of \eqref{A':blockdiag} can be done in an analogous way to the proof of Theorem \ref{H:thm}. These three properties can also be derived from the fact that the matrix representations of both $\nd(\gamma_{C,R})$ and $\nd(1)$ have the same three properties.
		
		Now proving \eqref{A':formula}. Let $n>0,m>0$ and for $x\in\rum{D}$ write $x \equiv x_1+ix_2$ for real-valued $x_1$ and $x_2$. Then the solution to \eqref{eq:pde} with $\gamma\equiv 1$ and Neumann boundary condition $f_n$ from \eqref{eq:fbasis} is
		\begin{equation*}
			w_n(x) \equiv \frac{1}{n\sqrt{2\pi}}x^n,\enskip n>0.
		\end{equation*}
		From \eqref{eq:frechet} we get
		\begin{equation*}
			A'_{n,m} = -\int_{B_{C,R}} \nabla w_m\cdot \overline{\nabla w_n}\,dx = -\frac{1}{\pi}\int_{B_{C,R}} x^{m-1}\overline{x}^{n-1}\,dx.
		\end{equation*}
		Applying a change of variables from $x$ to $x+C$ and writing $x=\eta e^{i\theta}$ in polar coordinates yields
		\begin{align*}
			A'_{n,m} &= -\frac{1}{\pi}\int_{B_{0,R}} (x+C)^{m-1}\overline{(x+C)}^{n-1}\,dx \\
			&= -\frac{1}{\pi}\int_0^R\int_{0}^{2\pi} (\eta e^{i\theta}+C)^{m-1}(\eta e^{-i\theta}+\overline{C})^{n-1}\eta\,d\theta\, d\eta.
		\end{align*}
		Now using the binomial theorem for $(\eta e^{i\theta}+C)^{m-1}$ and $(\eta e^{-i\theta}+\overline{C})^{n-1}$ 
		\begin{equation*}
			A'_{n,m} = -\frac{1}{\pi}\sum_{k=0}^{m-1}\sum_{k'=0}^{n-1} \binom{m-1}{k}\binom{n-1}{k'}C^{m-k-1}\overline{C}^{n-k'-1}\int_0^R \eta^{k+k'+1}\,d\eta\int_{0}^{2\pi} e^{i(k-k')\theta}\,d\theta.
		\end{equation*}
		Here the term $\int_{0}^{2\pi} e^{i(k-k')\theta}\,d\theta$ is only non-zero for $k'=k$, which can only hold for $k$ up to $\min\{n,m\}-1$. Also recalling that $C = ce^{i\sang}$ gives the expression in \eqref{eq:A'formula}
		\begin{align*}
			A'_{n,m} &= -2\sum_{k=0}^{\min\{n,m\}-1} \binom{m-1}{k}\binom{n-1}{k}C^{m-k-1}\overline{C}^{n-k-1}\int_0^R \eta^{2k+1}\,d\eta \\
			&= -e^{i(m-n)\sang}\sum_{k=0}^{\min\{n,m\}-1} \frac{1}{k+1}\binom{m-1}{k}\binom{n-1}{k}c^{m+n-2k-2}R^{2k+2}.
		\end{align*}
	\end{proof}
	
\end{proposition}

Analogous to \eqref{centrohermstruct} the matrix structure of $A'$ is 
\begin{equation*}
	A' = \begin{pmatrix}
		\mathcal{J}\roverline{A'^+}\mathcal{J} & 0 \\
		0 & A'^+
	\end{pmatrix},
\end{equation*}
thus we only need to evaluate the lower right part $A'^+$, and as $A'$ is Hermitian it is sufficient to evaluate the upper triangular part of $A'^+$.

\section{Implementation details and numerical results} \label{sec:numerical}

In this section we will shortly discuss the implementation details for the algorithms \eqref{eq:regnonlinrecon} and \eqref{eq:reglinrecon}, and apply the linear and non-linear approach to the three examples in \fref{fig:fig2}. These three examples are difficult scenarios for EIT reconstruction. Example~A will demonstrate if the algorithms can reconstruct very non-convex shapes, in particular where the non-convex part is oriented away from the closest boundary. Example~B will test if the algorithms can separate relatively small convex inclusions. Since the monotonicity method cannot detect holes in inclusions \cite{Harrach13}, the point of Example~C will be to test whether one wide inclusion can partially shield another inclusion, potentially making it difficult to separate the two in the presence of noise. The $\beta$-values for the reconstructions are chosen as in \eqref{eq:betaval} which for the examples in \fref{fig:fig2} are $\beta^{\textup{nonlin}} = 4$ and $\beta^{\textup{lin}} = 0.8$.
\begin{figure}[!htb]
\centering
\subfigure[Example A]{\includegraphics[width=.32\textwidth,natwidth=6.29in,natheight=5.29in]{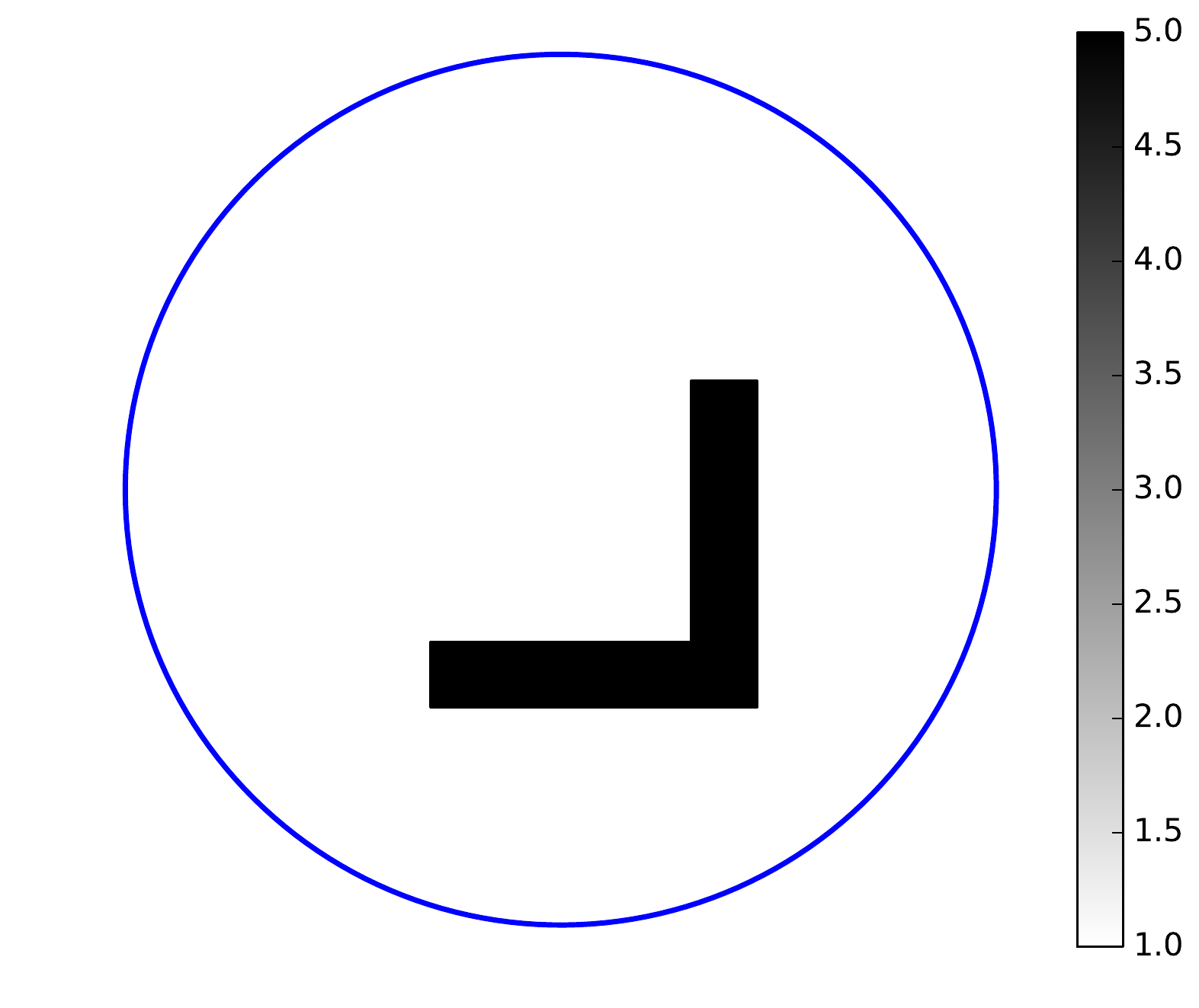}} \label{fig:fig2a}
\hfill
\subfigure[Example B]{\includegraphics[width=.32\textwidth,natwidth=6.29in,natheight=5.29in]{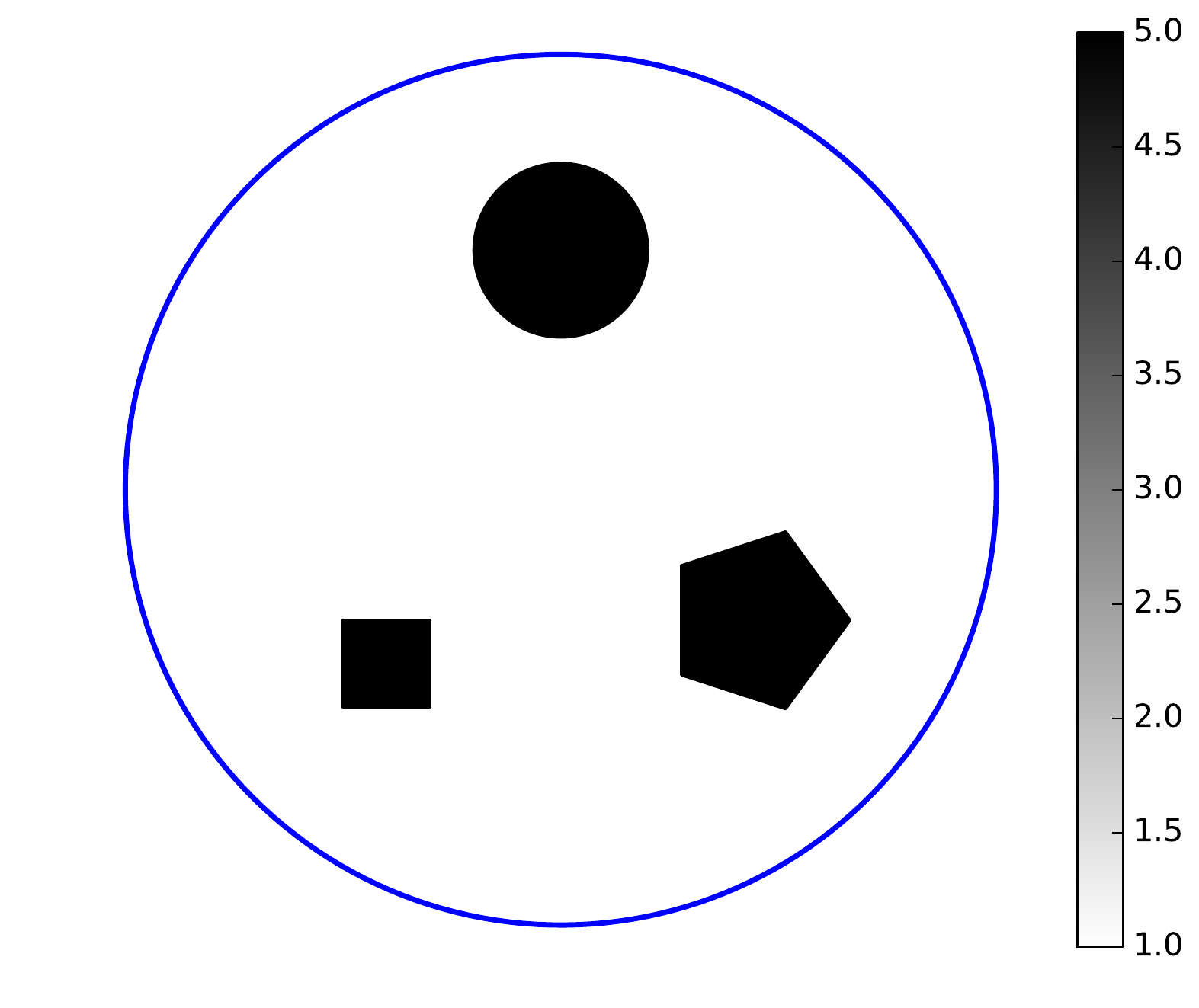}} \label{fig:fig2b}
\hfill
\subfigure[Example C]{\includegraphics[width=.32\textwidth,natwidth=6.29in,natheight=5.29in]{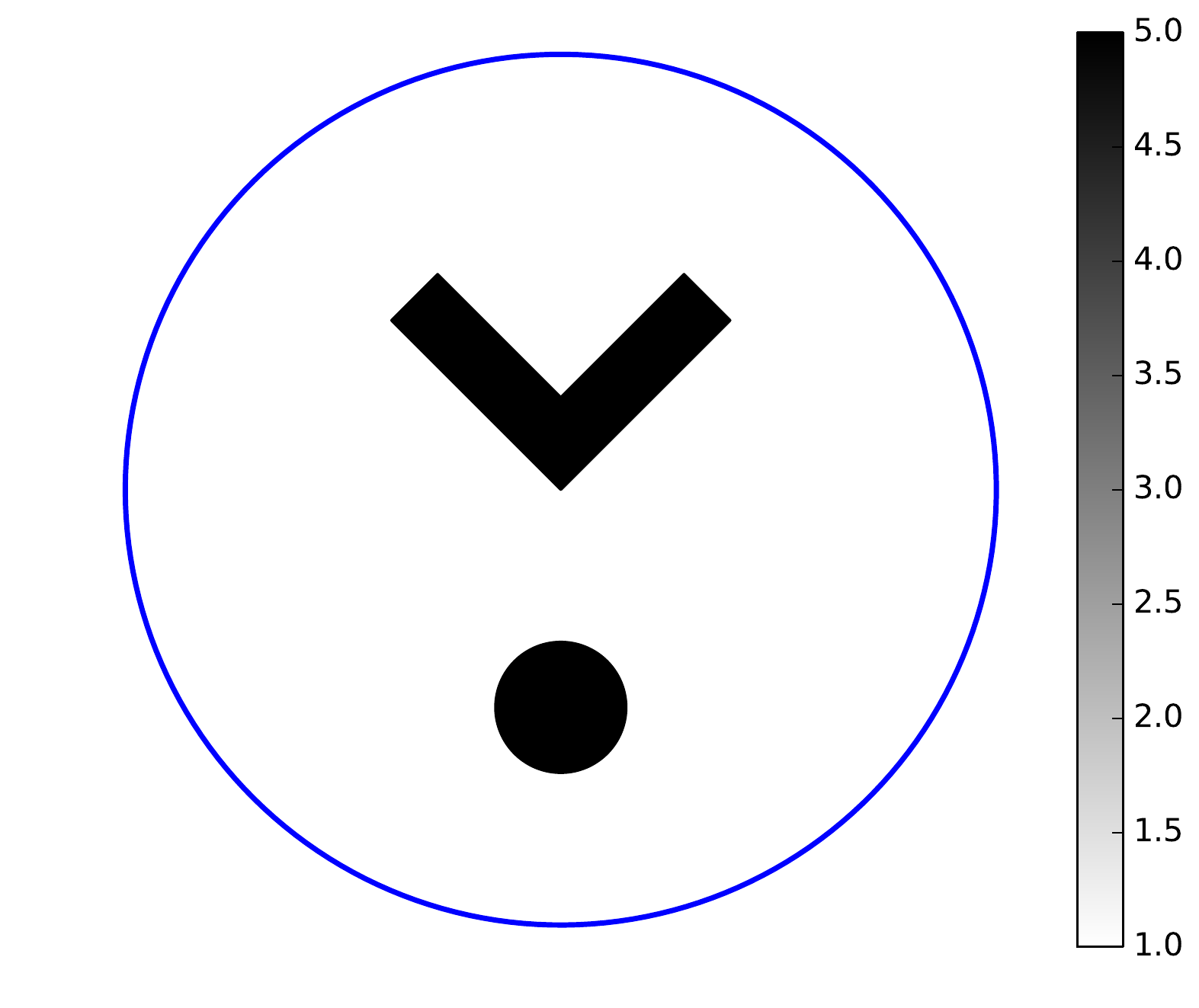}} \label{fig:fig2c}
\caption{Numerical phantoms.} \label{fig:fig2}
\end{figure}

For the numerical implementation we will use a regular hexagonal tiling, where each hexagon has a centre $C$ and radius $R$ such that its corners are placed on the boundary of the ball $B_{C,R}$. Here the resolution is controlled by the radius $R$ which is kept fixed; it is chosen as $R = 0.025$ for the given examples. For the monotonicity-based reconstructions \eqref{eq:regnonlinrecon} and \eqref{eq:reglinrecon} a hexagon is included in the reconstruction if the corresponding ball $B_{C,R}$ yields a positive semi-definite operator in the monotonicity test. The positive semi-definiteness is determined from the sign of the smallest eigenvalue, which are real-valued as all the involved operators are self-adjoint. 

There can be hexagons inside the inclusions that are not included in the reconstruction when the slightly larger balls intersect the inclusion boundaries. There will always be issues near the boundaries when using a finite discretization. Since the main focus of this paper is the comparison of the two methods then using a regular tiling, rather than overlapping balls, makes the plots visually easier to read and compare.   

In practice we can only use finite dimensional approximations to the matrix representations in \eqref{afact} and \eqref{a'mat}, so we have $\absm{n},\absm{m} = 1,2,\dots, N$ which gives $2N\times 2N$ matrices. For any linear compact operator $\mathcal{F}: L^2_\diamond(\partial\rum{D})\to L^2_\diamond(\partial\rum{D})$ the corresponding $N$-term matrix approximation is
\begin{equation*}
	F^N_{n,m} \equiv \inner{\mathcal{F}f_m,f_n},\enskip \absm{n},\absm{m}=1,2,\dots, N. 
\end{equation*}
As $\{f_n\}_{n\in\rum{Z}\setminus\{0\}}$ is an orthonormal basis for $L^2_\diamond(\partial\rum{D})$ then $F^N$ is a matrix representation of $P_N \mathcal{F}P_N$, where $P_N$ is the orthogonal projection onto $\mspan\{ f_n \}_{\absm{n}=1,2,\dots,N}$. It is straightforward to show that the eigenvalues of $F^N$ and $P_N \mathcal{F}P_N$ coincide, and from spectral theory of compact operators (e.g.\ \cite{Osborn1975,Kato1995}) it is well known that the spectrum of $P_N \mathcal{F}P_N$ converge to that of $\mathcal{F}$ as $N\to\infty$. 

In the numerical examples below we use $N = 16$ which implies 32 orthonormal current patterns; more than 32 current patterns for a 2D reconstruction is often considered excessive. Since the factorization \eqref{afact} holds in terms of infinite matrices, we use a much larger $\tilde{N}>N$ to generate the $H_a$-matrices for which we according to \eqref{Amatstruct} only construct the $\tilde{N}\times\tilde{N}$ matrices $H_a^+$. Afterwards we use \eqref{Amatstruct} to construct the larger $2\tilde{N}\times 2\tilde{N}$ matrix $A^{\tilde{N}}$ and extract the central $2N\times 2N$ matrix $A^N$ to use for the monotonicity tests in the non-linear case. For the following examples $\tilde{N} = 200$ was used, and that sufficient accuracy is attained is checked through the involution property of $H_a$; namely how large a central $2N\times 2N$ part of $H_a^{\tilde{N}} H_a^{\tilde{N}}$ that equals an identity matrix. 

From Theorem \ref{H:thm}.\eqref{H:submat} it is observed that $(H_\rho)_{n,m}$ is a polynomial in $\rho$ with at most $m$ non-zero terms. The coefficients of the polynomial are independent of $a$, and can therefore be precomputed and reused for the evaluation of $H_a$ for each $a$. Since the coefficients in the polynomial are binomial coefficients, and that we need to evaluate them up to a high index $\tilde{N} = 200$, the summation in \eqref{Hkj} quickly becomes numerically unstable. For this purpose the Python library gmpy2 \cite{gmpy2} is used, which has a fast implementation for exact evaluation of binomial coefficients, and has data structures that support much higher precision (in terms of no.\ of digits) and is able to accurately evaluate the expressions in \eqref{Hkj}. Alternatively, a more stable approach which would not require gmpy2 or the equivalent, is to apply Gauss-Legendre quadrature to the inner products \eqref{Hamat} with weights $\mathbf{w}$ and sample points $\boldsymbol{\theta}$, which from Lemma \ref{lemma:psia} gives
\begin{equation}
	(H_a)_{n,m} \simeq \frac{(-1)^n e^{-in\sang}}{2\pi} \mathbf{w}^{\textup{T}}\exp \left[ i\left( m\boldsymbol{\theta}-2n\arctan\left( \frac{1+\rho}{1-\rho}\tan\left[ \frac{1}{2}(\boldsymbol{\theta}-\sang) \right] \right) \right) \right]. \label{eq:quadint}
\end{equation}
It should be noted that using \eqref{eq:quadint} is about 30 times slower than \eqref{Hkj} when the binomial coefficients are reused, in order to attain the same precision for $\tilde{N} = 200$.

The involved PDEs to simulate the data $\nd(\gamma)$, for the three examples in \fref{fig:fig2}, are solved with a finite element method with piecewise affine elements. The applied mesh is excessively fine ($1.3\cdot 10^5$ nodes) and is aligned with the inclusions, such that we can expect that only the applied noise and the finite dimensional truncation $2N$ has significant influence in the comparison. 

The added noise is a matrix $E^\delta$ cf. \eqref{eq:ndnoisy}, which is constructed in the following way: for each index $(n,m)$ let $E^1_{n,m}$ be a realization from a normal $\mathcal{N}(0,1)$-distribution, and take the Hermitian and centrohermitian parts:
\begin{equation}
	E^2 \equiv \frac{1}{2}\left[E^1 + (E^1)^*\right], \quad E^3 \equiv \frac{1}{2}\left[E^2 + \mathcal{J}\roverline{E^2}\mathcal{J}\right].  \label{eq:symerr}
\end{equation}
If $\mathcal{A}$ is the matrix representation of the noiseless datum $\nd(\gamma)$, then we scale 
\begin{equation*}
	E^4_{n,m} \equiv E^3_{n,m}\mathcal{A}_{n,m}.
\end{equation*}
Finally, the noise is scaled to have a specified norm $\delta$ in the operator norm
\begin{equation*}
	E^\delta \equiv \frac{\delta}{\norm{E^4}}E^4.
\end{equation*}
A different noise realization is used for each noise level $\delta$, however the same noisy data $\nd^\delta(\gamma)$ is used for both the linear and non-linear reconstruction. It should be noted that the measured data $\nd^\delta(\gamma)$ can in practice be symmetrized to achieve \eqref{eq:symerr}, which often reduces the noise level significantly below $\delta$. How much the noise level is reduced completely depends on the particular noise realization. The way the noise is added in this paper makes the reconstructions less dependent on whether we were \emph{lucky} that part of the noise cancels out. The reconstructions are in this sense a worst-case scenario for a specified noise level $\delta$.

It was proved in \cite[Theorem 1]{GardeStaboulis_2016} that the regularization parameter could be chosen as ${\alpha = \delta}$, however it is no guarantee for the best choice of regularization parameter. Furthermore, there is also the truncation of the dimension to $2N$ which implies that slightly more regularization is required. In \cite[Section 5]{GardeStaboulis_2016} it was suggested that a good choice of regularization parameter in both linear and non-linear cases is
\begin{equation}
	\alpha = -\mu \inf\sigma(\nd(1)-\nd^\delta(\gamma)), \label{eq:regchoice}
\end{equation}
where $\sigma$ denotes the spectrum of the operators, and $\mu$ is a parameter that must be tuned; typically very close to 1. The values of $\mu$ are chosen manually, and all lie in the range $[0.9986,1.0002]$.

The reconstruction is fast as it only requires computation of eigenvalues for $2N\times 2N$ Hermitian matrices, and it is suited for parallel computing as the monotonicity tests for different balls can be done completely independently. When $R=0.025$ and without utilizing parallel computing, each reconstruction takes an average of $0.22$ seconds on a laptop with an Intel i7 processor with CPU clock rate of 2.4 Ghz.

In the discretized setting
\begin{equation*}
	S \equiv \{B_{C,R} : (C,R)\text{-hexagon in the discretization}\},
\end{equation*}
we denote the reconstructions as $S_\alpha \equiv S\cap \mathcal{T}_\alpha$ and $S_\alpha'\equiv S\cap \mathcal{T}_\alpha'$ (cf.\ \eqref{eq:regnonlinrecon} and \eqref{eq:reglinrecon}). In this sense, we can define absolute and relative differences with respect to which balls/hexagons that are not in common in the linear and non-linear reconstructions. More precisely, if $|S|$ denotes the number of elements in $S$ (and likewise for the sets below) the absolute and relative differences are defined as
\begin{equation*}
	e_{\text{abs}} \equiv \absm{S_\alpha\setminus S_\alpha'} + \absm{S_\alpha'\setminus S_\alpha}, \qquad e_{\text{rel}} \equiv \frac{e_{\text{abs}}}{|S|}.
\end{equation*}
From \fref{fig:lin} and \ref{fig:nonlin} it is clear that there is hardly any difference in the reconstructions based on the linear and non-linear methods. This is in accordance with table~\ref{tab:diff}, where for the most part there is a difference of 0--3 hexagons; most of which lead to a smallest eigenvalue with a magnitude close to machine precision, and can therefore be attributed to rounding errors.

\begin{table}[htb]
\begin{tabularx}{\textwidth}{|c?Y|Y?Y|Y?Y|Y?}
\cline{2-7}
	\multicolumn{1}{c?}{} & \multicolumn{2}{c?}{Example A} & \multicolumn{2}{c?}{Example B} & \multicolumn{2}{c?}{Example C} \\
	\hline
	$\delta$  & $e_{\text{abs}}$ & $e_{\text{rel}}$ & $e_{\text{abs}}$ & $e_{\text{rel}}$ & $e_{\text{abs}}$ & $e_{\text{rel}}$ \\
	\hline
	\rule{0pt}{2.5ex}$0$       & 10 & $5.432\cdot 10^{-3}$ & 1 & $5.432\cdot 10^{-4}$ & 2 & $1.086\cdot 10^{-3}$ \\
	$10^{-5}$ & 3 & $1.630\cdot 10^{-3}$ & 2 & $1.086\cdot 10^{-3}$ & 3 & $1.630\cdot 10^{-3}$ \\
	$10^{-4}$ & 1 & $5.432\cdot 10^{-4}$ & 0 & 0 & 1 & $5.432\cdot 10^{-4}$ \\
	$10^{-3}$ & 2 & $1.086\cdot 10^{-3}$ & 0 & 0 & 0 & 0 \\
	$10^{-2}$ & 2 & $1.086\cdot 10^{-3}$ & 2 & $1.086\cdot 10^{-3}$ & 1 & $5.432\cdot 10^{-4}$ \\
	\hline
\end{tabularx}
\caption{Absolute and relative differences between the linear and non-linear reconstructions for the examples in \fref{fig:fig2} with different levels of noise $\delta$.} \label{tab:diff}
\end{table}

It is observed that for 32 current patterns it is difficult to reconstruct non-convex shapes, in particular the large L-shaped inclusion in Example~A where the non-convex part is pointed away from the closest boundary. For the non-convex inclusion in Example~C the reconstruction is reasonable in the noiseless case, where for increased noise the separation of the two inclusions is lost. This is a common feature of the monotonicity reconstructions when one larger inclusion partially shields another.

For Example~B there is a reasonable separation of the inclusions even for the highest noise level, and both shapes and locations are found well in the cases $\delta = 0, 10^{-5}, 10^{-4}$. A slight positioning error is present in Example B even in the noiseless case, this can sometimes happen when there are multiple inclusions of various sizes placed asymmetrically. The positioning error as well as shape errors for low noise levels are mainly due to the dimensional truncation; these errors can be reduced by either increasing the pixel size (this affects the size of the balls in the test inclusions) or by increasing the number of current patterns used.

\begin{figure}[htb]
\centering
\includegraphics[width=\textwidth,natwidth=20.1in,natheight=27.21in]{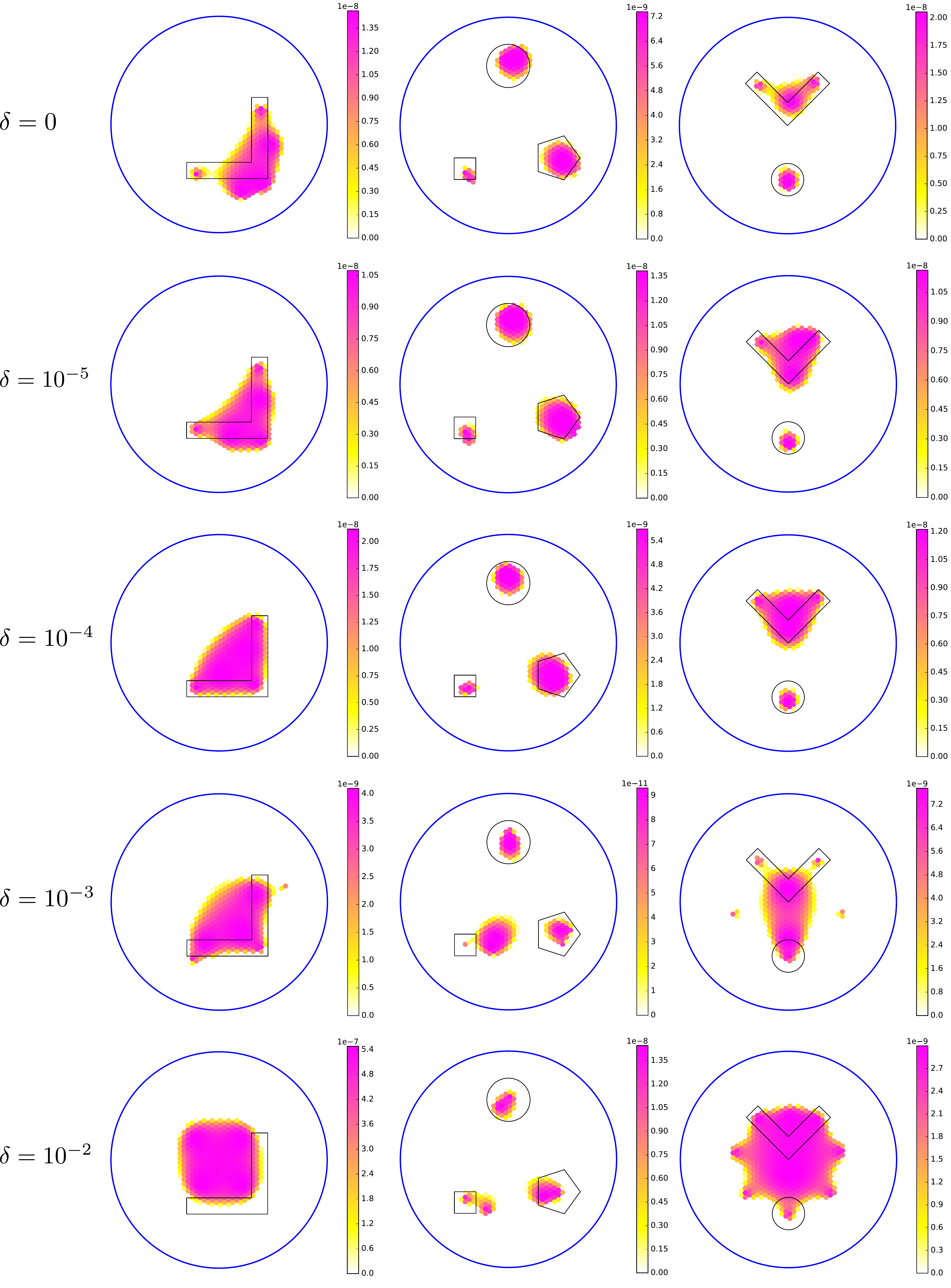}
\caption{Monotonicity reconstruction using the linear algorithm for the examples in \fref{fig:fig2} with various levels of noise $\delta$. The smallest eigenvalues are plotted for the hexagons where the inclusions are detected. The correct targets are outlined with a solid black line.} \label{fig:lin}
\end{figure}

\begin{figure}[htb]
\centering
\includegraphics[width=\textwidth,natwidth=20.1in,natheight=27.21in]{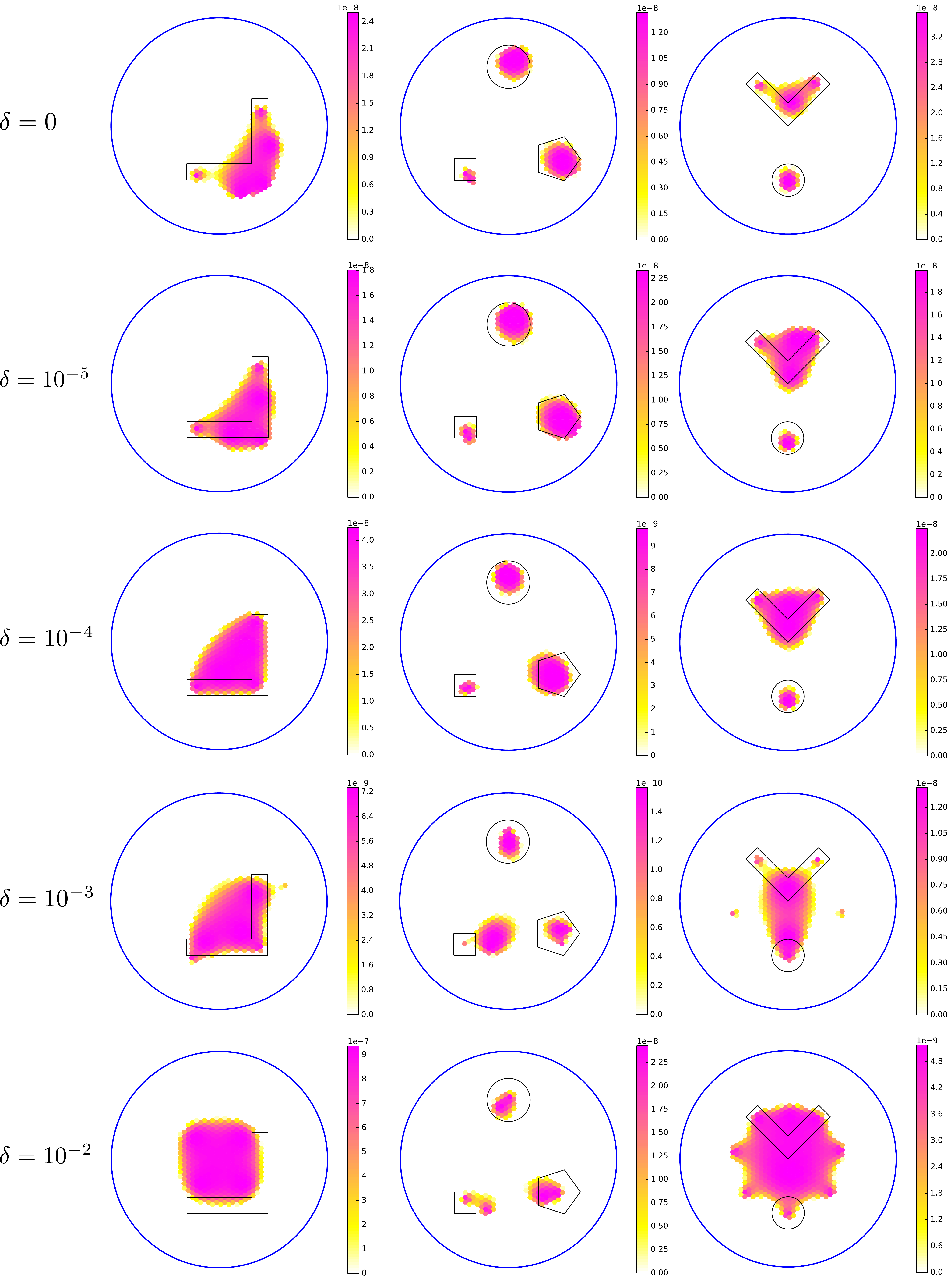}
\caption{Monotonicity reconstruction using the non-linear algorithm for the examples in \fref{fig:fig2} with various levels of noise $\delta$. The smallest eigenvalues are plotted for the hexagons where the inclusions are detected. The correct targets are outlined with a solid black line.} \label{fig:nonlin}
\end{figure}

\section{Conclusions}\label{sec:conclusions}

The linear and non-linear algorithms for monotonicity-based shape reconstruction in electrical impedance tomography are compared, and surprisingly found to essentially yield the same reconstructions both for noiseless and for noisy data. Exact matrix characterizations are derived for the Neumann-to-Dirichlet map and its Fr\'echet derivative for the ball inclusions used in the monotonicity tests. These matrix characterizations ensure that the sources of errors in the reconstructions are limited to the finite dimensional truncation and the added noise.

It is clear that the monotonicity method performs best for detecting small convex shapes, and here it is often possible to separate inclusions quite well in the presence of noise. For non-convex shapes one usually obtain something that resembles a convex approximation to the shape, either due to noise or the limited number of current patterns. 

\section*{Funding}

This research is supported by Advanced Grant No.\ 291405 HD-Tomo from the European Research Council.

\clearpage
\bibliographystyle{gIPE}
\bibliography{minbib}

\end{document}